\documentclass[12pt]{article}
\usepackage[a4paper, margin=1in]{geometry}
\usepackage[utf8]{inputenc}
\usepackage{braket}
\usepackage{bbm}
\usepackage{tikz}
\usepackage{amssymb}
\usepackage{amsmath}
\usepackage{mathrsfs}
\usepackage{mathtools}
\usepackage{graphicx}
\usepackage{enumerate}
\newcommand{\Mod}[1]{\ (\mathrm{mod}\ #1)}

\usepackage{amsthm}
\newtheorem{theorem}{Theorem}[section]
\newtheorem{lemma}[theorem]{Lemma}
\newtheorem{corollary}[theorem]{Corollary}

\newtheorem{proposition}[theorem]{Proposition}

\usepackage{lipsum}

\makeatletter
\newcommand{\subjclass}[2][2022]{%
  \let\@oldtitle\@title%
  \gdef\@title{\@oldtitle\footnotetext{#1 \emph{Mathematics subject classification.} #2}}%
}
\newcommand{\keywords}[1]{%
  \let\@@oldtitle\@title%
  \gdef\@title{\@@oldtitle\footnotetext{\emph{Key words and phrases.} #1.}}%
}

\title{Uniform distribution of polynomially-defined additive function to varying moduli}
\author{Agbolade Akande}
\date{\today}

\begin{document}

\maketitle
\section{Introduction}
Let $q$ be a positive integer and $f$ be an integer-valued arithmetic function. We say $f$ is uniformly distributed modulo $q$ if for every residue class $a\Mod{q}$
\begin{eqnarray*}
\#\{n \le x:f(n) \equiv a\Mod{q}\} \sim \frac{x}{q} & \text{as $x \rightarrow \infty$}.
\end{eqnarray*}
One way to prove that $f$ is uniformly distributed is by considering the mean value of $f$ composed with additive characters modulo $q$ \cite[Proposition 1.1]{Narkiewicz}:
\begin{center}
    $f$ is uniformly distributed modulo $q$ if and only if for all $r = 1,2,\hdots,q-1$, one has $\lim \limits_{x \rightarrow  \infty}\frac{1}{x}\sum \limits_{n \le x} \exp\left(\frac{2\pi i rf(n)}{q} \right) = 0$.
\end{center}
Pillai in 1940\cite{Pillai} used this approach to show $\Omega(n)$ is uniformly distributed modulo $q$ for every positive integer $q$. $\Omega(n)$ is an example of an \textbf{additive} function: an integer-valued arithmetic function $f$ with the property $f(mn) = f(m)+f(n)$ for coprime $m$ and $n$. In 1969, Delange gave a general criterion \cite{Delange} for when additive functions $f$ are uniformly distributed modulo $q$ for fixed $q$. Delange's criterion is as follows:
\begin{theorem}\label{Delange}
Define two conditions for every integer $t > 1$:
\begin{center}
    $(A_t)$ The series $\sum \limits_{\underset{t \nmid f(p)}{p}} p^{-1}$ diverges,
\end{center}
and
\begin{center}
    $(B_t)$ For all $k \ge 1$, the number $\frac{2f(2^k)}{t}$ is an odd integer.
\end{center}
An additive function $f$ is uniformly distributed modulo $q$ in the following 3 cases:
\begin{itemize}
    \item $q$ is odd and $A_p$ holds for every prime divisor $p$ of $q$.
    \item $2 \mid \mid q$, $A_p$ holds for every odd prime divisor $p$ of $q$ and either $A_2$ or $B_2$ hold.
    \item $4\mid q$, $A_p$ holds for every odd prime divisor $p$ of $q$, and either $A_2$ holds or both $B_2$ and $A_4$ hold.
\end{itemize}
Moreover, these conditions are necessary for $f$ to be uniformly distributed modulo $q$.
\end{theorem}
\noindent
We are interested in proving uniform distribution of a class of additive functions where the modulus is allowed to vary. To illustrate, consider the sum of prime divisors function $A(n) = \sum \limits_{p^k \mid \mid n} kp$. It was proved by Pollack and Singha Roy that for a fixed $K \ge 1$ and any residue class $a \mod{q}$,
\begin{eqnarray*}
\sum \limits_{\overset{n \le x}{A(n) \equiv a \Mod{q}}}1 \sim \frac{x}{q} & \text{as $x \rightarrow \infty$,}
\end{eqnarray*}
for $q \le (\log x)^K$. This was a byproduct of their work on the distribution of multiplicative functions in coprime residue classes\cite{PaulAkash}.\\
\\
In this paper, we will consider \textbf{polynomially-defined additive} functions $f$: Additive functions $f$ such that $f(p) = F(p)$ for all primes $p$, where $F(x)$ is a nonconstant polynomial in $\mathbb{Z}[x]$. In the next section, given a polynomially-defined additive function $f$ satisfying mild conditions, we explicitly describe the collection of integers $q$ such that $f$ is uniformly distributed modulo $q$. We will then set up the same framework used in \cite{PaulAkash} by Pollack and Singha Roy. (See also the paper of Katai \cite{Katai}, where a broadly similar approach is taken.) A new ingredient is the use of estimates for the size of sets of the form:
$$\mathcal{V}_{F,q,J}(w) = \left\{(v_1,\cdots,v_J) \mod{q}: (v_1,\hdots,v_J,q) = 1,\sum \limits_{j=1}^JF(v_j) \equiv w \Mod{q} \right\}.$$
These estimates require bounds on exponential sums adopted from Cochrane \cite{Expsum} and Loh \cite{Residue}. The end result is as follows:
\begin{theorem}\label{Main Theorem}
Fix $K \ge 1$ and $\epsilon \in (0,1]$. Let $f$ be a polynomially-defined additive function such that $f(p) = F(p)$ for all primes $p$ where $F(x) \in \mathbb{Z}[x]$ with degree $d$. Let $\mathcal{S}_f$ be the set of all positive integers $q$ such that $f$ is uniformly distributed modulo $q$. Then as $x \rightarrow  \infty$,
$$\sum \limits_{\overset{n \le x}{f(n) \equiv a\Mod{q}}}1 \sim \frac{x}{q},$$
for all residue classes $a\mod{q}$ and $q \in \mathcal{S}_f \cap [1,(\log x)^K]$ provided that either
\begin{itemize}
    \item $d = 1$, or
    \item $q \le (\log x)^{\left(1 - \epsilon\right)\left(1-\frac{1}{d}\right)^{-1}}$.
\end{itemize} 
\end{theorem}
\noindent
As we now explain, the extra condition on $q$ in Theorem \ref{Main Theorem} when $d > 1$ is essentially optimal. Fix an integer $d \ge 2$ and define $F(x) = (x-1)^d + 1$. Let $f$ be the completely additive function where $f(p) = F(p)$ for all primes $p$ and let $q = q_1^d$. By Theorem \ref{Delange}, $f$ is uniformly distributed for all $q$. Notice that $F(p) \equiv 1 \Mod{q}$ if and only if $p \equiv 1 \Mod{q_1}$. Hence, by Siegel-Walfisz's theorem, there are at least $\frac{1 + o(1)}{\varphi(q_1)}\frac{x}{\log x}$ primes $p \le x$ with $F(p) \equiv 1 \Mod{q}$. Therefore, the residue class $1 \mod q$ is overrepresented if $\frac{1}{\varphi(q_1)}\frac{x}{\log x}$ is much bigger than $\frac{x}{q}$, for example if $q^{1-1/d} > (\log x )^{(1+\delta)}$ for some fixed $\delta > 0$. This means that for uniform distribution we want $q$ to be no more than $\approx (\log x)^{(1-\frac{1}{d})^{-1}}$. So the condition is essentially optimal.
\section{Polynomial-defined arithmetic functions that are uniformly distributed modulo $q$}
Given a polynomially-defined additive function, we will use Delange's criterion in Theorem \ref{Delange} to determine the collection of integers $q$ such that $f$ is uniformly distributed modulo $q$. To check condition $B_2$, we will impose the restriction: $f(2^k)$ is odd for all $k \ge 2$.
\begin{proposition}\label{Class of integers}
Fix $q \in \mathbb{N}\backslash\{1\}$. Let $f$ be a polynomially-defined additive function, such that $f(p) = F(p)$ for all primes $p$, $F(x) \in \mathbb{Z}[x]$. Suppose $f(2^k)$ is an odd integer for all $k \ge 2$. If there exist an odd prime $p$ such that $p \mid q$ and $F(x) \mod{p}$ has a factor of
$$\prod \limits_{i=1}^{p-1} (x-i),$$
then $f$ is not uniformly distributed modulo $q$.\\
\\
If $2 \mid \mid q$ and $F$ has an even constant and $F(x) \mod{2}$ has a factor of $(x-1)$, then $f$ is not uniformly distributed modulo $q$.\\
\\
If $4 \mid q$ and the alternating sum of coefficients of F is either $0$ or $2\mod{4}$, then $f$ is not uniformly distributed modulo $q$.\\
\\
If neither of these conditions hold, then $f$ is uniformly distributed modulo $q$.
\end{proposition}
\begin{proof}
First, let $p$ be an odd prime divisor $p$ of $q$. By Theorem \ref{Delange}, $f$ is not uniformly distributed modulo $q$ if $A_p$ does not hold. Now we will use the fact that, for all $d \ge 2$,
\begin{center}
    $A_d$ holds if and only if there exist $a$ such that $(a,d) = 1$ and $f(a) \not \equiv 0\Mod{d}$ 
\end{center}
This follows from Dirichlet's theorem on primes in arithmetic progressions. Therefore, the only polynomials where $A_p$ does not hold are the polynomials $F$ such that $F \mod{p}$ has a factor of $\prod \limits_{i=1}^{p-1} (x-i)$.\\
\\
Suppose $2 \mid \mid q$. By Theorem \ref{Delange}, $f$ is not uniformly distributed modulo $q$ if $A_2$ and $B_2$ both fail. The $A_2$ case is similar to the last proof: the polynomials where $A_2$ does not are the polynomials with the sum of its coefficients equal to 0 modulo 2. This is equivalent to polynomials $F$ where $F \mod{2}$ has a factor of $(x-1)$. The $B_2$ case fails when the constant term is even.\\
\\
Suppose $4\mid q$. By Theorem \ref{Delange}, $f$ is not uniformly distributed modulo $q$ if $A_4$ does not hold. Let $F(x) = \sum \limits_{i=0}^d a_ix^i$ where $a_i \in \mathbb{Z}$. If $F(1) \equiv F(3) \equiv 0 \Mod{4}$. Then $\sum \limits_{\overset{0 \le i \le d}{\text{$i$ even}}} a_i +\sum \limits_{\overset{0 \le i \le d}{\text{$i$ odd}}} a_i \equiv 0 \Mod 4$ and $\sum \limits_{\overset{0 \le i \le d}{\text{$i$ even}}} a_i - \sum \limits_{\overset{0 \le i \le d}{\text{$i$ odd}}} a_i \equiv 0 \Mod 4$. The only solutions to these system of congruences are $\sum \limits_{\overset{0 \le i \le d}{\text{$i$ even}}} a_i \equiv \sum \limits_{\overset{0 \le i \le d}{\text{$i$ odd}}} a_i \equiv 0 \Mod 4$ and $\sum \limits_{\overset{0 \le i \le d}{\text{$i$ even}}} a_i \equiv \sum \limits_{\overset{0 \le i \le d}{\text{$i$ odd}}} a_i \equiv 2 \Mod 4$. \\
\\
If $F$ does not satisfy any of those conditions mentioned, the criteria hold for all the cases in Theorem \ref{Delange}. Thus, $f$ is uniformly distributed modulo $q$.
\end{proof}
\section{Bounds on polynomials composed with additive characters}
We will use bounds from Cochrane \cite{Expsum} and Loh \cite{Residue}. The first bound in \cite{Expsum} applies to polynomials that are nonconstant modulo $\ell$ for $\ell$ prime; a polynomial $F$ is \textbf{nonconstant modulo $\ell$} if and only if the map $F:\mathbb{Z}/\ell\mathbb{Z} \mapsto \mathbb{Z}/\ell\mathbb{Z}$ is nonconstant.
\begin{lemma}\label{Cochrane}
For all primes $\ell$ and polynomials $F(x) \in \mathbb{Z}[x]$ that are nonconstant modulo $\ell$,
$$\left \lvert \sum \limits_{\overset{v \mod \ell^k}{(v,\ell) = 1}} {\exp\left(\frac{2 \pi i F(v)}{\ell^k} \right)} \right \rvert \le 8.82 \ell^{k(1-\frac{1}{d+1})}. $$
\end{lemma}
\noindent
Lemma \ref{Cochrane} cannot be applied to all polynomials we are interested in. For example, let $F(x) = x^2+x+1$ and $f$ be the completely additive function such that $f(p) = F(p)$ for all primes $p$. By Theorem \ref{Delange}, $f$ is uniformly distributed modulo $q$ for all positive integers $q$ but $F$ is constant modulo 2. So, in the rest of this section, we will focus on the polynomials Lemma \ref{Cochrane} does not apply to. \\
\\
Firstly, notice that $\left \lvert \sum \limits_{\overset{v \mod \ell^k}{(v,\ell) = 1}} {\exp\left(\frac{2 \pi i F(v)}{\ell^k} \right)} \right \rvert = \left \lvert \sum \limits_{\overset{v \mod \ell^k}{(v,\ell) = 1}} {\exp\left(\frac{2 \pi i (F(v)+t)}{\ell^k} \right)} \right \rvert $ for all $t \in \mathbb{Z}$ so the constant term of $F$ does not play a role. Also, if $F$ is constant modulo $\ell$ with $d \ge 1$, there have to be at least 2 terms in the polynomial. \\
\\
Next, we will consider the polynomials where $\ell$ does not divide the content of $F$: We define the \textbf{content of $F$} to be the greatest common divisor of the coefficients of $F$. The following lemma, which gives an upper bound on the exponential sum with a condition on the content of $F$, is stated in \cite{Residue}. 
\begin{lemma}\label{Loh}
Let $F(x)$ be a polynomial in $\mathbb{Z}[x]$ such that
$$F(x) = a_sx^{d_s} + a_{s-1}x^{d_{s-1}} + \cdots+ a_2x^{d_2} + a_1x^{d_1} $$
with $s \ge 2$ and $d_s > d_{s-1} > \cdots > d_1 \ge 1$, and $a_i \in\mathbb{Z}\backslash \{0\}$.\\
\\
Suppose $\ell$ is prime and $k$ is a positive integer.\\
\\
If $\ell > d_s$ and $\ell$ does not divide the content of $F$, then
$$\left \lvert \sum \limits_{\overset{v \mod \ell^k}{(v,\ell) = 1}} {\exp\left(\frac{2 \pi i F(v)}{\ell^k} \right)} \right \rvert \le (d_s - 1)\ell^{m(1-\frac{1}{s})}.$$
If $\ell \le d_s$ and $\ell$ does not divide the content of $F$, then
$$\left \lvert \sum \limits_{\overset{v \mod \ell^k}{(v,\ell) = 1}} {\exp\left(\frac{2 \pi i F(v)}{\ell^k} \right)} \right \rvert \le R(\ell^{-\tau}F') \ell^{\frac{\tau + 1}{s}}\ell^{m(1-\frac{1}{s})}$$
where $\tau$ is the largest nonnegative integer such that $\ell^\tau$ divides the content of $F'(x)$ and $R(F)$ is the number of roots of the congruence
$$F(X) \equiv 0\Mod{\ell}.$$
\end{lemma}
\noindent
Since $\tau$ and $R(\ell^{-\tau} F')$ depend only on $F$, they are both $O_F(1)$. Also, if $d$ is the degree of $f$, then $s < d+1$ . Hence, the inequalities in Lemma \ref{Loh} can be combined into a single inequality so that a uniform upper bound on the exponential sum is $O_F(\ell^{m(1-\frac{1}{d+1})})$. Still, there are polynomials that Lemma \ref{Cochrane} and Lemma \ref{Loh} cannot be applied to, the polynomials $F$ where $\ell$ does divide the content of $F$ without its constant term and $F$ is constant modulo $\ell$. For instance, $F(x) = 2x+1$ is constant modulo 2 and 2 divides the content of $F(x) - 1$ but the completely additive function $f$ such that $f(p) = F(p)$ for all primes $p$ is uniformly distributed modulo $q$ for all positive integers $q$.\\
\\
Let $t$ be the largest nonnegative integer such that $\ell^t$ divides the content of $F(x)$ without its constant term $c_0$. Then, $F(x) = \ell^tG(x) + c_0$ where $G(x) \in \mathbb{Z}[x]$ and $\ell$ does not divide the content of $G(x)$. Then for $k > t$,
$$\left \lvert \sum \limits_{\overset{v \mod \ell^k}{(v,\ell) = 1}} {\exp\left(\frac{2 \pi i F(v)}{\ell^k} \right)} \right \rvert =\left \lvert \sum \limits_{\overset{v \mod \ell^k}{(v,\ell) = 1}} {\exp\left(\frac{2 \pi i G(v)}{\ell^{m-t}} \right)} \right \rvert = \ell^{t}\left \lvert \sum \limits_{\overset{v \mod \ell^{m-t}}{(v,\ell) = 1}} {\exp\left(\frac{2 \pi i G(v)}{\ell^{m-t}} \right)} \right \rvert. $$
We now apply the bounds that we already have on the last sum. Since $\ell^t$ divides the content of $F$ without its constant term, $\ell^t$ is $O_F(1)$. Combining all the bounds that we have, we can form a corollary:
\begin{corollary}\label{Corollary}
Let $F(x)$ be a nonconstant polynomial in $\mathbb{Z}[x]$. Then, for all positive integers $k$ and primes $\ell$,
$$\left \lvert \sum \limits_{\overset{v \mod \ell^k}{(v,\ell) = 1}} {\exp\left(\frac{2 \pi i F(v)}{\ell^k} \right)} \right \rvert \le  C_F\ell^{k(1-\frac{1}{d+1})}$$
where $C_F$ is a constant that depends on $F$.
\end{corollary}

\section{Proof of Theorem \ref{Main Theorem}: the notion of convenient numbers}
 Fix $K\ge 1$ and $\epsilon > 0$. Let $f$ be a polynomially-defined additive function such that $f(p) = F(p)$ for all primes $p$ and $F(x) \in \mathbb{Z}[x]$. We will focus on integers $n \le x$ and varying $q \in \mathcal{S}_f \cap [1,(\log x)^K]$. Similar to Pollack and Singha Roy's paper\cite{PaulAkash}, we will consider the $n$ with many large non-repeated prime factors and we will show that the count of everything else is negligible compared to the count of these special integers. \\
 \\
 We will first define parameters $J = \lfloor \log \log \log x \rfloor$ and $y = \exp ((\log x)^\frac{\epsilon}{2})$. We call a positive integer $n$ \textbf{convenient} if the following conditions hold:
\begin{itemize}
    \item $n \le x$,
    \item the $J$ largest prime factors of $n$ with multiplicity exceed $y$,
    \item none of these $J$ primes are repeated in $n$.
\end{itemize}
To summarize, $n$ is convenient if and only if $n$ admits an expression: $n = mP_1\cdots P_J$, where $P_1,\hdots,P_J$ are primes with
\begin{eqnarray*}
    \max\{P(m),y\}< P_J < \cdots <P_1, & P_1\cdots P_J\le x/m.
\end{eqnarray*}
\begin{lemma}\label{Inconvenient}
The number of inconvenient $n\le x$ is at most $o(x)$.
\end{lemma}
\begin{proof}
We follow the proof in \cite[Lemma 3.1]{PaulAkash}. Suppose $n \le x$ is inconvenient. We can assume $P(n) > z = x^{1/\log \log x}$ ($P(n)$ is the largest prime factor of $n$) because the number of exceptional $n \le x$, by theorems on smooth numbers \cite[Theorem 5.13 and Corollary
5.19, Chapter III.5]{Smooth}, is $\frac{x}{(\log x)^{(1+o(1))\log \log \log x}} = o(x)$. We can also assume that $n$ has no repeated prime factors exceeding $y$ because the number of exceptional $n \le x$ is $O(x/y) = o(x)$.\\
\\
We will denote $n = PAB$ where $P = P(n)$ and $A$ is the largest divisor of $n/P$ supported on primes exceeding $y$. Then $z < P \le x/AB$. If $A$ and $B$ are given, the number of possibilities of $P$ is bounded by $\pi(x/AB) \ll \frac{x}{AB\log z} \ll \frac{x \log \log x}{AB \log x}$. We sum on $A$ and $B$ to bound the number of inconvenient $n$. As $n$ has no repeated primes exceeding $y$ and $n$ is inconvenient, $\Omega(A) < J$. Thus $\sum \frac{1}{A} \le \left(1+\sum \limits_{p \le x} \frac{1}{p} \right)^J \le (2 \log \log x)^J \le \exp\left(O((\log \log \log x)^2) \right)$. Since $B$ is $y$-smooth, $\sum \frac{1}{B} = \prod \limits_{p \le y} \left(\sum \limits_{j=0}^\infty \frac{1}{p^j} \right)\ll \exp \left(\sum \limits_{p \le y} \frac{1}{p} \right) \ll \log y= (\log x)^\frac{\epsilon}{2}$. Combining all the results, the number of the remaining inconvenient $n \le x$ is at most 
\begin{align*}
 \hfill   &\hfill & \frac{x \log \log x}{(\log x)^{1-\frac{\epsilon}{2}}}\exp(O((\log \log \log x)^2)) = o(x).&\hfill & \hfill&  \qedhere\\
\end{align*}
\end{proof}
\noindent
Next, we show that the count of inconvenient integers $n\le x$ with the condition $f(n) \equiv a \Mod{q}$, for polynomially-defined additive function $f$, is negligible compared to the count of convenient integers with the same condition:
\begin{lemma}\label{Inconvenient congruence}
Let $f$ be a polynomially-defined function, such that $f(p) = F(p)$ for all primes and $F(x)$ is a polynomial in $\mathbb{Z}[x]$ with degree $d$. If $q$ satisfies the conditions in Theorem \ref{Main Theorem}, the number of inconvenient $n \le x$ where $f(n) \equiv a \Mod{q}$ is $o(x/q)$ for all residue classes $a \mod q$.
\end{lemma}
\begin{proof}
We follow the the verification of Hypothesis B in section 5 of \cite{PaulAkash}. Suppose $n \le x$ is inconvenient and $f(n) \equiv a\Mod{q}$. From the proof in Lemma \ref{Inconvenient}, we can assume $P(n) > z = x^{1/\log \log x}$ and that $n$ has no repeated prime factors exceeding $y$ because the number of exceptional $n \le x$ is $o(x/q)$. \\
\\
We will denote $n = PAB$ where $P = P(n)$ and $A$ is the largest divisor of $n/P$ supported on primes exceeding $y$. Then $z < P \le x/AB$ and $f(P) + f(AB)\equiv a \Mod{q}$. Let $\xi(q)$ be the maximum number of roots $v \mod{q}$ of any congruence $F(v) \equiv a' \Mod{q}$, where the maximum is over all residue classes $a'\mod{q}$. By the Brun–Titchmarsh inequality, given $A$ and $B$ there are $\ll \xi(q)\frac{x}{\varphi(q)AB \log(z/q)} \ll \xi(q)\frac{x \log \log x}{\varphi(q)AB \log x} $ corresponding values of $P$. Summing over $A$ and $B$, using the same process as Lemma \ref{Inconvenient}, the number of remaining inconvenient $n \le x$ such that $f(n) \equiv a\Mod{q}$ is at most
$$\xi(q)\frac{x \log \log x}{\varphi(q)(\log x)^{1-\frac{\epsilon}{2}}}\exp(O((\log \log \log x)^2)).$$
Since $\frac{1}{\log \log 3q} \ll \frac{\varphi(q)}{q} \ll 1 $ for all $q \ge 2$, in our range of $q$ the ratio $\frac{\varphi(q)}{q}$ is negligible compared to any power of $\log x$. Hence, it is enough to show $\xi(q) = o\left((\log x)^{1-\frac{3}{4}\epsilon} \right)$.\\
\\
When $d = 1$, $\xi(q) = O_F(1)$. (In fact, it is bounded above by the absolute value of leading coefficient of $F$.) Hence, let $d \ge 2$. We will first consider when $q$ is prime. Then, $\xi(q)$ is at most $d$ expect for finitely many primes. (If $F$ does not reduce to a constant modulo $q$, $\xi(q)$ is at most $d$). Extending to when $q$ is squarefree, $\xi(q) = O_F(d^{\omega(q)})$. Using the asymptotic formula $\omega(q) \ll \frac{\log q }{\log \log q}$ from \cite{Hardy}, in our range of $q$, $d^{\omega(q)} = \exp\left(O_F\left(\frac{\log \log x \log d}{\log \log \log x} \right)\right)$ which is smaller than the bound we want. Otherwise, by a result of Konyagin in \cite{Konyagin1}, each congruence $F(v) \equiv a' \Mod{q}$ has $O_F(q^{1-1/d})$ roots modulo $q$ for any residue class $a'\mod{q}$. Using the condition in Theorem \ref{Main Theorem}, $q \le (\log x)^{\left(1-\epsilon\right)\left(1-1/d\right)^{-1}}$ will satisfy our bound.
\end{proof}
\noindent
Therefore, by Lemmas \ref{Inconvenient} and \ref{Inconvenient congruence}, it suffices to show
\begin{equation}\label{eq: 1}
\sum \limits_{\overset{\text{convenient $n \le x$}}{f(n) \equiv a \Mod{q}}}1 \sim \frac{1}{q}\sum \limits_{\text{convenient $n \le x$}}1. 
\end{equation}
We will now rewrite the number of ways we can count the number of convenient numbers as follows:
\begin{equation}\label{eq: convenient}
\sum \limits_{\text{convenient $n \le x$}}1 = \frac{1}{J!}\sum \limits_{m \le x} \sum \limits_{\overset{P_1,\hdots,P_J \text{ distinct}}{\overset{P_1 \cdots P_J \le x/m}{{\text{each } P_J > L_m}}}}1
\end{equation}
where $L_m = \max\{y,P(m)\}$. We will define for an arbitrary class $w \mod q$
$$\mathcal{V}(w) = \{(v_1,\cdots,v_J) \mod{q}: (v_1,\hdots,v_J,q) = 1,\sum \limits_{j=1}^JF(v_j) \equiv w \Mod{q} \}$$
so we can write
\begin{equation}\label{eq: 2}
\sum \limits_{\overset{\text{convenient $n \le x$}}{f(n) \equiv a \Mod{q}}}1 = \frac{1}{J!}\sum \limits_{m \le x} \sum \limits_{v \in \mathcal{V}_m} \sum \limits_{\overset{P_1,\hdots,P_J \text{ distinct}}{\overset{P_1 \cdots P_J \le x/m}{\overset{\text{each } P_J > L_m}{\text{each }P_j \equiv v_j \Mod{q}}}}}1    
\end{equation}
where $L_m = \max\{y,P(m)\}$ and $\mathcal{V}_m = \mathcal{V}(a-f(m))$. \\
\\
We will proceed as in \cite{PaulAkash} to remove the congruence in the innermost sum over the $J$ largest primes in \eqref{eq: 2}. Write
$$\sum \limits_{\overset{P_1,\hdots,P_J \text{ distinct}}{\overset{P_1 \cdots P_J \le x/m}{\overset{\text{each } P_j > L_m}{\text{each }P_j \equiv v_j \Mod{q}}}}}1 = \sum \limits_{\overset{P_2,\hdots,P_J \text{ distinct}}{\overset{P_2 \cdots P_J \le x/mL_m}{\overset{\text{each } P_j > L_m}{\text{each }P_j \equiv v_j \Mod{q}}}}}\sum \limits_{\overset{P_1 \ne P_2,\hdots,P_J \text{ distinct}}{\overset{L_m <P_1 \le x/(mP_2\cdots P_J)}{P_1 \equiv v_1 \Mod{q}}}}1.$$
Since $L_m \ge y$ and $q \le (\log x)^K = (\log y)^{2K}$, the Siegel-Walfisz theorem implies that
$$\sum \limits_{\overset{P_1 \ne P_2,\hdots,P_J \text{ distinct}}{\overset{L_m <P_1 \le x/(mP_2\cdots P_J)}{P_1 \equiv v_1 \Mod{q}}}}1 = \frac{1}{\varphi(q)}\sum \limits_{\overset{P_1 \ne P_2,\hdots,P_J \text{ distinct}}{{L_m <P_1 \le x/(mP_2\cdots P_J)}}}1 +O\left(\frac{x}{mP_2\cdots P_J} \exp(-C_K \sqrt{\log y} ) \right).$$
Thus,
$$\sum \limits_{\overset{P_1,\hdots,P_J \text{ distinct}}{\overset{P_1 \cdots P_J \le x/m}{\overset{\text{each } P_j > L_m}{\text{each }P_j \equiv v_j \Mod{q}}}}}1 = \frac{1}{\varphi(q)}\sum \limits_{\overset{P_1,\hdots,P_J \text{ distinct}}{\overset{P_1 \cdots P_J \le x/m}{\overset{\text{each } P_j > L_m}{\text{$(\forall j \ge 2)$ }P_j \equiv v_j \Mod{q}}}}}1 + O\left(\frac{x}{m}\exp\left(-\frac{1}{2}C_K(\log x)^{\epsilon/4} \right) \right).$$
Repeating the same steps for $P_2,\hdots,P_J$, we have
$$\sum \limits_{\overset{P_1,\hdots,P_J \text{ distinct}}{\overset{P_1 \cdots P_J \le x/m}{\overset{\text{each } P_j > L_m}{\text{each }P_j \equiv v_j \Mod{q}}}}}1 = \frac{1}{\varphi(q)^J}\sum \limits_{\overset{P_1,\hdots,P_J \text{ distinct}}{\overset{P_1 \cdots P_J \le x/m}{{\text{each } P_j > L_m}}}}1 + O\left(\frac{x}{m}\exp\left(-\frac{1}{4}C_K(\log x)^{\epsilon/4} \right) \right).$$
Putting this back into \eqref{eq: 2} and noting that $\#\mathcal{V}_m \le (\log x)^{KJ}$,
\begin{equation}\label{eq: 3}
\sum \limits_{\overset{\text{convenient $n \le x$}}{f(n) \equiv a \Mod{q}}}1 = \sum \limits_{m \le x} \frac{\#\mathcal{V}_m}{\varphi(q)^J} \left( \frac{1}{J!}\sum \limits_{\overset{P_1,\hdots,P_J \text{ distinct}}{\overset{P_1 \cdots P_J \le x/m}{{\text{each } P_j > L_m}}}}1 \right) + O\left(x \exp\left( -\frac{1}{8}C_K (\log x)^{\epsilon/4}\right)\right).
\end{equation}
We will find an asymptotic formula for $\#\mathcal{V}_m$ using exponential sums. We can write $\#\mathcal{V}_m = \#\mathcal{V}(w) = \prod \limits_{\ell^e \mid \mid q}\#V_{\ell^e}$ for each prime power $\ell^e \mid \mid q$ where $V_{\ell^e}$ can be written as follows:
$$V_{\ell^e} = \{(v_1,\cdots,v_J) \mod{\ell^e}: (v_1,\hdots,v_J,\ell) = 1,\sum \limits_{j=1}^JF(v_j) \equiv w \Mod{\ell^e} \}.$$
By the orthogonality relations for exponential sums,
$$\ell^e \#V_{\ell^e} = \sum \limits_{r \mod \ell^e}\sum \limits_{\overset{v_1,\hdots,v_J}{(v_1\cdots v_J,\ell)=1}}\exp\left( \frac{2\pi i r\left[\sum \limits_{j=1}^J F(v_j) -w \right]}{\ell^e}\right)$$
$$=\varphi(\ell^e)^J + \sum \limits_{0 < r < \ell^e}\exp\left(\frac{-2\pi irw}{\ell^e}\right) \left[\sum \limits_{\overset{v \mod \ell^e}{(v,\ell) = 1}}\exp\left(\frac{2 \pi i rF(v)}{\ell^e} \right)\right]^J$$
$$=\varphi(\ell^e)^J+\sum \limits_{k=1}^e (\ell^{J})^{e-k}\left( \sum \limits_{\overset{0<r<\ell^k}{\ell \nmid r}}\exp\left(\frac{-2\pi irw}{\ell^k}\right) \left[\sum \limits_{\overset{v \mod \ell^k}{(v,\ell) = 1}}\exp\left(\frac{2 \pi i rF(v)}{\ell^k} \right)\right]^J\right).$$
We can rewrite this as $\varphi(\ell^e)^J\left[1+\sum \limits_{k=1}^e H(\ell,k,J) \right]$, where 
$$H(\ell,k,J) = \sum \limits_{\overset{0<r<\ell^k}{\ell \nmid r}}\exp\left(\frac{-2\pi irw}{\ell^k}\right) \left[\sum \limits_{\overset{v \mod \ell^k}{(v,\ell) = 1}} \frac{\exp\left(\frac{2 \pi i r F(v)}{\ell^k} \right)}{\varphi(\ell^k)}\right]^J.$$
Let $\ell$ be a prime dividing $q$ with $e = e(\ell,q)$ satisfying $\ell^e \mid \mid q$. Define $S_\ell = \mathbb{N} \cap [1,e]$ and three subsets $X_\ell$, $Y_\ell$, and $Z_\ell$ such that
\begin{itemize}
    \item $X_\ell = \left\{k \in S_\ell: \left \lvert\sum \limits_{\overset{v \mod \ell^k}{(v,\ell) = 1}} \exp\left(\frac{2 \pi i rF(v)}{\ell^k} \right) \right \rvert = {\varphi(\ell^k)} \right\}$,
    \item $Y_\ell = \{k \in S_\ell\backslash X_\ell:  \frac{C_F\ell^{k(1-\frac{1}{d+1})}}{\varphi(\ell^k)} < 1\}$ where $C_F$ is the constant in Corollary \ref{Corollary}, and
    \item $Z_\ell = \{k \in S_\ell\backslash(X_\ell \cup Y_\ell)\}$.
\end{itemize}
We will show that for each $\ell$, the sums over $Y_\ell$ and $Z_\ell$ are vanishing as $J \rightarrow \infty$, that $X_\ell$ and $Z_\ell$ are empty for sufficiently large $\ell$, and that the asymptotic bounds on $\#\mathcal{V}(w)$ come from $X_\ell$ for small $\ell$.\\
\\
Define $M = M(F,q)$ to be the largest positive integer dividing $q$ such that $F(v) \equiv F(1) \Mod{M}$ for all $v$ coprime to $M$. $M$ is well defined and $M \ll_F 1$ because $F$ is a nonconstant polynomial. Let $D_\ell$ be the largest positive integer such that $\ell^{D_\ell} \mid M$.
\begin{lemma}\label{X_ell}
A closed form of the sum over $X_\ell$ is as follows:
$$\sum \limits_{k \in X_\ell} H(\ell,k,J) = \begin{cases} 0 & \text{if $\ell \nmid M$},\\ \ell^{D_\ell}-1 & \text{if $\ell \mid M$ and $F(1)J \equiv w \Mod{\ell^{D_\ell}}$},\\ -1 & \text{otherwise}. \end{cases}$$
\end{lemma}
\begin{proof}
If $k \in X_\ell$, $\exp\left(\frac{2 \pi i rF(v)}{\ell^k} \right) = \exp\left(\frac{2 \pi i rF(1)}{\ell^k} \right)$ for all $v$ coprime to $\ell$. Thus, $F(v) \equiv F(1) \Mod{\ell^k}$ and 
$\ell^k \mid M$. It follows that $X_\ell = \{ k \in \mathbb{N}: k \le D_\ell \}$; in particular, $X_\ell = \emptyset$ if $D_\ell = 0$. This gives the evaluation of the sum when $\ell \nmid M$. Otherwise,
$$\sum \limits_{k\in X_\ell}H(\ell,k,J) = \sum \limits_{k=1}^{D_\ell} \sum \limits_{\overset{0<r<\ell^k}{\ell \nmid r}}\exp\left(\frac{-2\pi ir(w - F(1) J)}{\ell^k}\right)  $$
$$= \sum \limits_{k=1}^{D_\ell}\left[\sum \limits_{r \mod{\ell^k}}\exp\left(\frac{-2\pi ir(w - F(1) J)}{\ell^k}\right) - \sum \limits_{r \mod{\ell^{k-1}}}\exp\left(\frac{-2\pi ir(w - F(1) J)}{\ell^{k-1}}\right) \right]$$
$$=-1+\left[\sum \limits_{r \mod{\ell^{D_\ell}}}\exp\left(\frac{-2\pi ir(w - F(1) J)}{\ell^{D_\ell}}\right)\right] .$$
By the orthogonality relation for additive characters, the sum on $r$ is $\ell^{D_\ell}$ if $F(1)J \equiv w\Mod{\ell^{D_\ell}}$ and 0 otherwise.
\end{proof}
\noindent
To study the sum over $Y_\ell$ and $Z_\ell$, we consider $B(\ell,k)= \frac{C_F\ell^{k(1-\frac{1}{d+1})}}{\varphi(\ell^k)}$ as a continuous function on the domain $\{(\ell,k):\text{$\ell \ge 2$ and $k \ge 1$}\}$. It is easy to see that $B(\ell,k)$ is a monotonically decreasing function of both $\ell$ and $k$. Thus, for every fixed $\ell$, $B(\ell,k) < 1$ for all large $k$. We record the following consequences:
\begin{itemize}
    \item For every fixed $\ell$, there exist a smallest positive integer $E_\ell$ not in $X_\ell$ such that $B(\ell,E_\ell) < 1$. Then $Y_\ell = \{k \in \mathbb{N}: E_\ell \le k \le e\}$.
    \item There exists a value of $k$, bounded in terms of $F$, with $B(\ell, k) < 1$ for all $\ell \ge 2$. Thus, $E_\ell \ll_F 1$.
    \item For large $\ell$, we have $B(\ell ,1) < 1$. Since $X_\ell$ is empty when $\ell > M$, $E_\ell = 1$ for all large $\ell$. For these $\ell$, we have $Y_\ell = S_\ell$ and $Z_\ell = \emptyset$. We fix $\ell_0 = \ell_0(F)$ such that $Y_\ell = S_\ell$ and $2C_F \le \ell^{\frac{1}{2(d+1)}}$ for all $\ell > \ell_0$.
\end{itemize}
If $k \not \in X_\ell$, $\left \lvert \sum \limits_{\overset{v \mod \ell^k}{(v,\ell) = 1}} \frac{\exp\left(\frac{2 \pi i r F(v)}{\ell^k} \right)}{\varphi(\ell^k)}\right \rvert  < 1 $. So we may find a positive $\eta \gg_F 1$, satisfying for all $\ell \le \ell_0$,
$$(1-\eta) \ge \begin{cases}
    B(\ell,k) & \text{if $k \ge E_\ell$},\\
    \left \lvert\sum \limits_{\overset{v \mod \ell^{k}}{(v,\ell) = 1}}\frac{\exp\left(\frac{2 \pi i r F(v)}{\ell^{k}} \right)}{\varphi(\ell^{k})}\right \rvert & \text{if $D_\ell < k < E_\ell$}.
\end{cases}$$
\begin{lemma}\label{Y_ell}
    As $x \rightarrow \infty$ (which coincides with $J \rightarrow \infty$),
    $$\sum \limits_{\ell \mid q} \left \lvert\sum \limits_{k \in Y_\ell} H(\ell,k,J) \right \rvert = o(1).$$
\end{lemma}
\begin{proof}
Using Corollary \ref{Corollary}, as $J \rightarrow \infty$,
$$\left \lvert \sum \limits_{k \in Y_\ell}H(\ell,k,J) \right \rvert \le \sum \limits_{k = E_\ell}^\infty \left[\frac{C_F\ell^{k(1-\frac{1}{d+1})}}{\varphi(\ell^k)}\right]^J \varphi(\ell^k) \ll \left[\frac{C_F\ell^{E_\ell(1-\frac{1}{d+1})}}{\varphi(\ell^{E_\ell})}\right]^J\ell^{E_\ell}\le \left[\frac{C_F\ell^{(1-\frac{1}{d+1})}}{\ell - 1}\right]^J\ell^{E_\ell}.$$
We used the fact that $\sum \limits_{k = E_\ell}^\infty \left[\frac{C_F\ell^{k(1-\frac{1}{d+1})}}{\varphi(\ell^k)}\right]^J \varphi(\ell^k)$ is a geometric series with ratio $\ell^{(1- \frac{J}{d+1})}$. As $J$ gets large, $\ell^{(1- \frac{J}{d+1})} \le 2^{(1- \frac{J}{d+1})} \le \frac{1}{2}$ so the sum of the geometric series is at most twice the first term.\\
\\
If we assume that $\ell > \ell_0$, then $E_\ell = 1$ and $\frac{C_F\ell^{(1-\frac{1}{d+1})}}{\ell - 1} \le \ell^{-\frac{1}{2(d+1)}}$ so that
\begin{equation}\label{eq: Y_ell}
\left \lvert \sum \limits_{k \in Y_\ell}H(\ell,k,J) \right \rvert = O\left( \ell^{1 - \frac{J}{2(d+1)}}\right) =  O(\ell^{-\frac{J}{4(d+1)}}).
\end{equation}
For the exceptional $\ell \le \ell_0$, as $J \rightarrow \infty$, 
\begin{equation}\label{eq: Y_ell part 2}
  \left \lvert \sum \limits_{k \in Y_\ell}H(\ell,k,J) \right \rvert \ll \left[\frac{C_F\ell^{E_\ell(1-\frac{1}{d+1})}}{\varphi(\ell^{E_\ell})}\right]^J\ell^{E_\ell} = O_F((1-\eta)^J)  
\end{equation}
recalling that $E_\ell \ll_F 1$. It follows from \eqref{eq: Y_ell} and \eqref{eq: Y_ell part 2} that $\sum \limits_{\ell \mid q} \left \lvert\sum \limits_{k \in Y_\ell} H(\ell,k,J) \right \rvert = o(1)$ as $J \rightarrow \infty$.
\end{proof}
\begin{lemma}\label{Z_ell}
As $x \rightarrow \infty$,
    $$\sum \limits_{\ell \mid q} \left \lvert\sum \limits_{k \in Z_\ell} H(\ell,k,J) \right \rvert = o(1).$$
\end{lemma}
\begin{proof}
Since $Z_\ell$ is empty for $\ell > \ell_0$, we may restrict attention to $\ell \le \ell_0$. By definition of $D_\ell$ and $E_\ell$, we can write $Z_\ell = \{ k \in \mathbb{N}: D_\ell < k < E_\ell\}$.  Thus, recalling that $E_\ell \ll_F 1$,
$$\left \lvert \sum \limits_{k\in Z_\ell}H(\ell,k,J) \right \rvert \le (E_\ell-D_\ell) \max \limits_{k \in Z_\ell} \left\{\left \lvert\sum \limits_{\overset{v \mod \ell^{k}}{(v,\ell) = 1}}\frac{\exp\left(\frac{2 \pi i r f(v)}{\ell^{k}} \right)}{\varphi(\ell^{k})}\right \rvert^J \varphi(\ell^{k}) \right\}$$
$$< (\ell^{E_\ell} \cdot E_\ell) \cdot\max \limits_{D_\ell< k < E_\ell} \left\{\left \lvert\sum \limits_{\overset{v \mod \ell^{k}}{(v,\ell) = 1}}\frac{\exp\left(\frac{2 \pi i r f(v)}{\ell^{k}} \right)}{\varphi(\ell^{k})}\right \rvert\right\}^J.$$
$$= O\left(\ell^{E_\ell} \cdot E_\ell \cdot (1-\eta)^J\right) = O_F((1-\eta)^J).$$ 
Hence, $\sum \limits_{\ell \mid q} \left \lvert\sum \limits_{k \in Z_\ell} H(\ell,k,J) \right \rvert = o(1)$ as $J \rightarrow \infty$.
\end{proof}
\noindent
We can write,
$$\#V_{\ell^e} = \frac{\varphi(\ell^e)^J}{\ell^e}\Big[1 +\sum \limits_{k \in X_\ell} H(\ell,k,J)+\sum \limits_{k \in Y_\ell} H(\ell,k,J)+\sum \limits_{k \in Z_\ell} H(\ell,k,J)\Big].$$
Multiplying over all $\ell$ dividing $q$,
$$\# \mathcal{V}(w) = \frac{\varphi(q)^J}{q} \prod \limits_{\ell \mid q}\left[ 1 +\sum \limits_{k \in X_\ell} H(\ell,k,J)+\sum \limits_{k \in Y_\ell} H(\ell,k,J)+\sum \limits_{k \in Z_\ell} H(\ell,k,J)\right].$$
To proceed, we observe that $\#\mathcal{V}(w) = 0$ unless $F(1) J \equiv w \Mod{\ell^{D_\ell}}$ for all $\ell$ dividing $M$. Indeed, since $\ell^{D_\ell}$ divides $M$, we have $F(v) \equiv F(1) \Mod{\ell^{D_\ell}}$ for all $v$ coprime to $\ell$. So if $(v_1,\hdots,v_J)$ is any element of $V_{\ell^e}$, then (keeping in mind that $D_\ell \le e$) it must be that
  $F(1)J \equiv \sum \limits_{j=1}^{J} F(v_j) \equiv w \Mod{\ell^{D_{\ell}}}$.  Thus, using Lemma \ref{X_ell},
$$\#\mathcal{V}(w)= \mathbbm{1}_{F(1)J \equiv w \Mod{M}}\frac{\varphi(q)^J}{q}\prod \limits_{\overset{\ell \mid q}{\ell \mid M}}\left[\ell^{D_\ell} + \sum \limits_{k \in Y_\ell \cup Z_\ell}H(\ell,k,J)  \right] \prod \limits_{\overset{\ell \mid q}{\ell \nmid M}}\left[1 + \sum \limits_{k \in Y_\ell \cup Z_\ell}H(\ell,k,J)   \right]$$
$$= \mathbbm{1}_{F(1)J \equiv w \Mod{M}}\frac{M\varphi(q)^J}{q}\prod \limits_{\overset{\ell \mid q}{\ell \mid M}}\left[1 + \frac{1}{\ell^{D_\ell}}\sum \limits_{k \in Y_\ell \cup Z_\ell}H(\ell,k,J)  \right] \prod \limits_{\overset{\ell \mid q}{\ell \nmid M}}\left[1 + \sum \limits_{k \in Y_\ell \cup Z_\ell}H(\ell,k,J)   \right].$$
By Lemmas \ref{Y_ell} and \ref{Z_ell}, $\sum \limits_{\ell \mid q}\left \lvert\sum \limits_{k \in Y_\ell \cup Z_\ell}H(\ell,k,J) \right \rvert = o(1)$ as $x \rightarrow \infty$. Therefore,
$$\prod \limits_{\overset{\ell \mid q}{\ell \mid M}}\left[1 + \frac{1}{\ell^{D_\ell}}\sum \limits_{k \in Y_\ell \cup Z_\ell}H(\ell,k,J)  \right]\prod \limits_{\overset{\ell \mid q}{\ell \nmid M}}\left[1 + \sum \limits_{k \in Y_\ell \cup Z_\ell}H(\ell,k,J)   \right] = 1+o(1),$$ 
and
$$\#\mathcal{V}(w) = (1+o(1))\mathbbm{1}_{F(1)J \equiv w \Mod{M}}\frac{M\varphi(q)^J}{q}.$$
Hence, using this asymptotic in \eqref{eq: 3},
$$\sum \limits_{\overset{\text{convenient $n \le x$}}{f(n) \equiv a \Mod{q}}}1 = \sum \limits_{m \le x} \frac{\#\mathcal{V}_m}{\varphi(q)^J} \left( \frac{1}{J!}\sum \limits_{\overset{P_1,\hdots,P_J \text{ distinct}}{\overset{P_1 \cdots P_J \le x/m}{\overset{\text{each } P_J > L_m}{F(1)J \equiv w\mod{M}}}}}1 \right) + O\left(x \exp\left( -\frac{1}{8}C_K (\log x)^{\epsilon/4}\right)\right) $$
$$=(1+o(1)) \frac{M}{q}\left( \frac{1}{J!} \sum \limits_{m \le x}  \sum \limits_{\overset{P_1,\hdots,P_J \text{ distinct}}{\overset{P_1 \cdots P_J \le x/m}{\overset{\text{each } P_J > L_m}{f(P_1\cdots P_J) \equiv a-f(m)\mod{M}}}}}1 \right) + O\left(x \exp\left( -\frac{1}{8}C_K (\log x)^{\epsilon/4}\right)\right).$$
Here we used the fact that $F(1)J \equiv w \mod{M}$ can be rewritten as $f(P_1\cdots P_J)\equiv a - f(m) \Mod{M}$ because $F(v) \equiv F(1) \Mod{M}$ for all $v$ coprime to $M$. Since $M \ll_F 1$, $P_j$ is coprime to $M$ for all $1 \le j \le J$ (for $y$ large). By \eqref{eq: convenient}, using the fact that $M \ll_F 1$,
\begin{equation}\label{eq: almost}
    \sum \limits_{\overset{\text{convenient $n \le x$}}{f(n) \equiv a \Mod{q}}}1 = \frac{M}{q}\sum \limits_{\overset{ \text{convenient }n \le x}{f(n) \equiv a \Mod{M}}}1 + o\left( \frac{x}{q}\right).
\end{equation}
Since $q \in \mathcal{S}_f$, $M \mid q$ and $M \ll_F 1$, Theorem \ref{Delange} implies 
\begin{equation}\label{eq: final Delange}
    \sum \limits_{\overset{n \le x}{f(n) \equiv a \Mod{M}}} 1 \sim \frac{1}{M}\sum \limits_{n \le x}1.
\end{equation}
Using Lemma \ref{Inconvenient} and the fact that $1\le M \ll_F 1$, 
\begin{align*}
    \sum \limits_{\overset{n \le x}{f(n) \equiv a \Mod{M}}} 1 = \sum \limits_{\overset{ \text{convenient }n \le x}{f(n) \equiv a \Mod{M}}}1 + o\left(\frac{x}{M} \right) & \text{    
    and} & \frac{1}{M}\sum \limits_{n \le x}1 =  \frac{1}{M}\sum \limits_{\text{convenient $n \le x$}}1 + o\left(\frac{x}{M} \right).
\end{align*}
Thus, by \eqref{eq: final Delange}, $\sum \limits_{\overset{ \text{convenient }n \le x}{f(n) \equiv a \Mod{M}}}1 = \frac{1}{M}\sum \limits_{\text{convenient $n \le x$}}1 + o\left(\frac{x}{M} \right)$, which along with \eqref{eq: almost} proves \eqref{eq: 1} and Theorem \ref{Main Theorem} follows.
\bibliographystyle{unsrt}
\bibliography{Dissertation/ReferencesF}
\end{document}